\documentclass[letter,10pt]{article}
\usepackage{amssymb}
\usepackage{amsthm}
\usepackage{amsmath}
\usepackage{mathrsfs}
\usepackage{verbatim}
\usepackage{enumerate}

\newtheorem{theorem}{Theorem}
\newtheorem{corollary}[theorem]{Corollary}
\newtheorem{lemma}[theorem]{Lemma}
\newtheorem{proposition}[theorem]{Proposition}
\newtheorem{remark}[theorem]{Remark}
\newtheorem{definition}[theorem]{Definition}
\newtheorem{example}[theorem]{Example}
\usepackage{cite}
\usepackage{hyperref}

\usepackage{color}

\begin{document}
\title{Generalized barycenters and variance maximization on metric spaces \footnote{ B.P. is pleased to acknowledge the support of  National Sciences and Engineering Research Council of Canada Discovery Grant number 04658-2018.}}

\author{ Brendan Pass\footnote{Department of Mathematical and Statistical Sciences, 632 CAB, University of Alberta, Edmonton, Alberta, Canada, T6G 2G1 pass@ualberta.ca.}}
\maketitle
\begin{abstract}
We show that the variance of a probability measure $\mu$ on a compact subset $X$ of a complete metric space $M$ is bounded by the square of the circumradius $R$ of the canonical embedding of $X$ into the space $P(M)$ of probability measures on $M$, equipped with the Wasserstein metric.  When barycenters of measures on $X$ are unique (such as on CAT($0$) spaces), our approach shows that $R$ in fact coincides with the circumradius of $X$ and so this result extends a recent result of Lim-McCann from Euclidean space.  Our approach involves bi-linear minimax theory on $P(X) \times P(M)$ and extends easily to the case when the variance is replaced by very general moments.

As an application, we provide a simple proof of Jung's theorem on CAT($k$) spaces, a result originally due to Dekster and Lang-Schroeder.

\end{abstract}

\section{Introduction}
The aim of this note is to establish bounds on the variance, and generalizations of it, of probability measures on metric spaces.   Letting $X \subseteq M$ be a compact subset of a complete metric space $(M,d)$, the variance of a probability measure $\mu \in P(X)$ on $X$ is
\begin{equation}\label{eqn: variance}
Var(\mu):=\inf_{y \in M}\int_Xd^2(x,y)d\mu(x).
\end{equation}	
Points $y$ where the infimum is attained (if they exist) are often called \emph{barycenters}, or Frechet means, of $\mu$. 	When $M =\mathbb{R}^n$, $\bar x_\mu =\int_Xxd\mu(x)$ is the unique barycenter, and so \eqref{eqn: variance}  coincides with the familiar formula $Var(\mu) =\int_X|x-\bar x_\mu|^2d\mu(x)$.  On the real line (ie, when $n=1$), a classical result of Popoviciu bounds the variance by the diameter of $X$: if $X \subseteq [a,b]$, then $Var(\mu) \leq \frac{1}{4}(b-a)^2$ for all $\mu \in P(X)$.  Equality occurs only when $\mu=\frac{1}{2}(\delta_a+\delta_b)$ \cite{Popoviciu35}.

Recent work by Lim and McCann extends this to higher dimensions \cite{LimMcCann20a}; for $X \subseteq \mathbb{R}^n$, they prove (among other results) that 
$$
Var(\mu) \leq R^2
$$
where $R$ is the radius of the smallest closed ball $\bar B(y,R)$ containing $X$ (in our subsequent terminology, $R$ will be called the \emph{circumradius}, $y$ the \emph{circumcenter} and $\bar B(y,R)$ the \emph{circumball}, of $X$).  Equality is obtained if and only if $\mu$ is supported on the intersection $\partial \bar B(y,R) \cap X$ of the corresponding sphere with $X$, and the barycenter of $\mu$ is the circumcenter $y$.  Their proof of this uses convex-concave minimax theory for the functional $H(y,\mu) =\int_X|x-y|^2d\mu(x)$.  This theory applies since the domain $\mathbb{R}^n$ of $y$ is affine and the integrand convex in $y$; Lim and McCann in fact prove an analogous result for more general convex moments (that is, when the integrand $|x-y|^2$ is replaced by a general convex function $v(x-y)$).  It does not appear, however, that one can directly apply these techniques to move beyond the affine setting and deal with more general metric spaces $M$ (or, even for $M=\mathbb{R}^n$, with non-convex integrands $V(x,y)$ replacing $|x-y|^2$ in \eqref{eqn: variance}).

We show here that this obstacle can be overcome by working with a type of \emph{generalized barycenter} introduced in recent work with Kim \cite{KimPass18}.  These amount to barycenters of measures on the image $i(X)$ of $X$ under the canonical isometric embedding $i:x \mapsto \delta_x$ of $X$ into the space $P(M)$ of probability measures on $M$, equipped with the Wasserstein metric.  Viewed within this framework, the analogue of the functional $H$ is \emph{bi-linear} on $P(X) \times P(M)$ (see \eqref{eqn: def of objection} below), without \emph{any} restriction on $M$.  We show that for any measure $\mu \in P(X)$ we have 
$$
Var(\mu) \leq R^2
$$
where $R$ is the circumradius of $i(X)$ within the metric space $P(M)$, with equality if and only if the pushforward $i_\# \mu$ is concentrated on the boundary of a circumball  of $i(X)$ and a Wasserstein barycenter $\nu$ of it is a circumcenter. In addition, since the bi-linearity does not depend on the integrand, this technique applies equally well to other continuous moment functions, yielding analogous results, as we show.

As an application of their work, Lim and McCann establish that among all measures $\mu \in P(X)$  supported on  sets $X \subseteq \mathbb{R}^n$ with diameter at most $1$, the maximal variance is obtained by uniform measure on the $n+1$ vertices of the unit simplex.  They show that this is equivalent to a classical theorem of Jung, which states that the diameter of any set $X$ is at most $2R\sqrt{\frac{n+1}{2n}}$, where $R$ is the circumradius of $X$ \cite{Jung1901}.  On the other hand,  Jung's theorem has been generalized to sufficiently small sets in metric spaces with upper sectional curvature bounds, first by Dekster \cite{Dekster97} under several technical assumptions, and subsequently, with these assumptions eliminated, by Lang and Schroeder \cite{LangSchroeder97}; we use our framework to provide a simple new proof of this result.

In the following section, we introduce the general setting we will work in and establish a preliminary result which underlies the rest of the paper.  In section 3 we prove our main result on variance maximizers on metric spaces, several consequences of it, and a generalization to general valuation or moment functions.  The fourth section is reserved for our new proof of Jung's theorem on metric spaces with upper curvature bounds.

\section{General setting and preliminary result}
In this section, we establish a background proposition which our subsequent analysis will hinge on.

Consider a continuous function $V:X \times Y \rightarrow \mathbb{R}$ on compact metric spaces $X$ and $Y$.  In the following sections, our primary interest will be in the case where $V$ is the squared distance on a common metric space $M$ containing both $X$ and $Y$; however, working with a general moment function $V$ costs no extra difficulty in the proofs in this section, and so we choose to do so in order to illustrate the flexibility of this approach.

Given a probability measure $\mu \in P(X)$, we define the \emph{$V$-variance} of $\mu$ as

$$
Var_V(\mu) =\min_{y \in Y}\int_XV(x,y)d\mu(x).
$$

We consider the problem of maximizing the $V$-variance over $P(X)$. We will show below that this is equivalent to minimizing the \emph{$V$-anti-variance} of $\nu$ over $P(Y)$, where the anti-variance of  $\nu \in P(Y)$ is defined as
$$
AVar_V(\nu) = \max_{x \in X}\int_YV(x,y)d\nu(y).
$$
We will refer to minimizers of  $y \mapsto \int_XV(x,y)d\mu(y)$ as \emph{$V$-barycenters} of $\mu$ and maximizers of $x \mapsto \int_YV(x,y)d\nu(y)$ as \emph{$V$-anti-barycenters of $\nu$}.

We will consider the following functional on $P(X) \times P(Y)$: 
\begin{equation}\label{eqn: def of objection}
\mathcal{V}(\mu,\nu):=\int_X\int_YV(x,y)d\mu(x)d\nu(y)
\end{equation}

We next note a simple fact about $\mathcal{V}$:
\begin{lemma}\label{lem: gen barycenters}
The $V$-variance of $\mu \in P(X)$ satisfies $Var_V(\mu) =\min_{\nu \in P(Y)} \mathcal{V}(\mu,\nu)$ and $\nu \in P(Y)$ minimizes $\nu \mapsto \mathcal{V}(\mu,\nu)$ if and only $\nu$ almost every $y$ is a $V$-barycenter of $\mu$.

Similarly, the  $V$-anti-variance of $\nu \in P(Y)$ satisfies $AVar_V(\nu) =\max_{\mu \in P(X)} \mathcal{V}(\mu,\nu)$,  and $\mu \in P(X)$ maximizes $\mu \mapsto \mathcal{V}(\mu,\nu)$ if and only $\mu$ almost every $x$ is a $V$-anti-barycenter of $\nu$. 
\end{lemma}
\begin{proof}
	We prove only the first assertion, since the proof of the second is essentially identical.
	
	For every $y$, we have $Var_V(\mu) \leq \int_XV(x,y)d\mu(x)$, with equality if and only if $y$ is a $V$-barycenter of $\mu$.  Integrating against any $\nu \in P(Y)$ yields $Var_V(\mu) \leq \mathcal{V}(\mu,\nu)$, with equality if and only if $\nu$ almost every $y$ is a $V$-barycenter, in which case $\nu$ minimizes  $\nu \mapsto \mathcal{V}(\mu,\nu)$.
\end{proof}

 Motivated by the preceding lemma, we will call minimizers of  $ \nu \mapsto \mathcal{V}( \mu, \nu)$ (respectively, maximizers of $ \mu \mapsto \mathcal{V}( \mu,\nu)$) \emph{generalized $V$-barycenters} of $\mu$ (respectively, \emph{generalized $V$-anti-barycenters} of $\nu$). 
 
 Note that when $X \subseteq M$, 
 $$
 Y=conv(X):=\{y:\text{ y minimizes } z\in M \mapsto \int_Xd^2(x,z)d\mu(x) \text{ for some } \mu \in P(X)   \}
 $$ is the \emph{barycentric convex hull of X} (that is, the set of all barycenters in $M$ of measures supported on $X$)  and $V(x,y) =d^2(x,y)$ is the metric distance squared, $V$-barycenters and the $V$-variance coincide with the classical metric barycenters and variance of the measure $\mu$ on the metric space $M$, respectively (equaling $\bar x_\mu=\int_Xxd\mu(x)$ and $\int_X|x-\bar x_\mu|^2d\mu(x)$  when $X\subseteq\mathbb{R}^n$  ).
In this case we will drop the prefix $V$ on the other nomenclature as well (so that, for example, a generalized barycenter of $\mu$ is a minimizer of $\nu \mapsto \int_X\int_Yd^2(x,y)d\mu(x)d\nu(y)$). The concept of generalized barycenters was in fact introduced in earlier joint work with Kim \cite{KimPass18} (although this terminology was not used there).

The following definition will play an important role in this paper.
\begin{definition}\label{def: saddle point}
	A pair $(\mu,\nu) \in P(X) \times P(Y)$ is called a \emph{saddle point} of $\mathcal{V}$ if $\mu$ is  a generalized $V$-anti-barycenter of $\nu$ and $\nu$ is a generalized $V$-barycenter of $\mu$.
\end{definition}

\begin{remark}
Lim-McCann's original motivation for variance maximization problems came from models describing swarming in physics and biology \cite{LimMcCann20b}; their goal was to understand minimizers of an interaction energy $\mu \mapsto \int_X\int_XV(x,y)d\mu(x)d\mu(y)$ of a population $\mu$ with pairwise interaction $V(x,y)$ (this is equivalent to the variance when $V(x,y) =|x-y|^2$ on $X=\mathbb{R}^n$; more generally, in a certain limit another term forces the population to live on a set with a fixed upper bound on its diameter).  In our setting, $\nu$ and $\mu$ might be thought of as populations of two species; agents $x \in X$ wish to maximize their average interaction energies, $\int_YV(x,y)d\nu(y)$ with agents in $Y$ while agents $y$ in $Y$ seek to minimize their average interaction energy $\int_YV(x,y)d\mu(x)$.  Saddle points capture exactly equilibria in this model; if $(\mu,\nu)$ is a saddle point, Lemma \ref{lem: gen barycenters} implies that $\mu$ almost every $x \in X$ is a $V$-anti-barycenter of $\nu$ while $\nu$ almost every $y$ is a $V$-barycenter of $\mu$.
\end{remark}

Establishing the following result is the main purpose of this section.
\begin{proposition}\label{thm: saddle point}
There exists a  maximizer $\mu$ of the $V$-variance among all measures in $P(X)$ and a minimizer $\nu$ among measures in $P(Y)$ of the $V$-anti-variance such that $(\mu,\nu)$ is a saddle point of $\mathcal{V}$.

Furthermore, any other $\bar \mu \in P(X)$ maximizes the $V$-variance if and only if 
$(\bar \mu,\nu)$ is a saddle point of $\mathcal{V}$. Similarly, any other $\bar \nu \in P(Y)$ minimizes the $V$-anti-variance if and only if $(\mu,\bar \nu)$ is a saddle point of $\nu$.
\end{proposition}

\begin{proof}
The proof is by standard minimax arguments. Note that (uniform on compact sets) continuity of $V$ implies weak continuity of the mappings $\mu \mapsto \mathcal{V}(\mu,\nu)$ and $\nu \mapsto \mathcal{V}(\mu,\nu)$ and thus existence of maximizers and minimizers respectively on the weakly compact sets $P(X)$ and $P(Y)$.  We define the set valued mapping $\mathcal{F}$ on $P(X) \times P(Y)$ by 
$$
\mathcal{F}:(\mu,\nu) \mapsto argmax \Big(\bar \mu \mapsto \mathcal{V}(\bar \mu,\nu)\Big) \times argmin \Big(\bar \nu \mapsto \mathcal{V}( \mu,\bar \nu)\Big)
$$
It is straightforward to show that $\mathcal{F}$ has a closed graph and that $\mathcal{F}(\mu,\nu)$ is non-empty, compact and convex for each $(\mu,\nu) \in P(X) \times P(Y)$.  The Kakutani -- Glicksberg -- Fan fixed point theorem (see, for instance, \cite{GranasDugundji03}) then ensures the existence of a fixed point $(\mu,\nu)$ of $\mathcal{F}$; that is, a saddle point of $\mathcal{V}$. 
 To see that $\mu$ maximizes the $V$-variance, note that for any other measure $\bar \mu \in P(X)$ we have, by Lemma \ref{lem: gen barycenters},
\begin{eqnarray*}
Var_V(\mu) &=&\int_X\int_YV(x,y)d\mu(x)d\nu(y)\\
&\geq&\int_X\int_YV(x,y)d\bar \mu(x)d\nu(y)\\
&\geq &Var_V(\bar \mu),
\end{eqnarray*}
with equality if and only $\bar \mu$ is a generalized $V$-anti-barycenter of $\nu$ (which gives equality in the first inequality above) and $\nu$ is a generalized $V$-barycenter of $\bar \mu$ (which gives equality in the second); that is, $(\bar \mu, \nu)$ is a saddle point. This yields the characterization of other maximizers in the statement of the Theorem.

A similar argument shows that $\nu$ minimizes the $V$-anti-variance, and yields the desired characterization of other minimizers $\bar \nu$ of the $V$-anti-variance.
\end{proof}
Although the proof of Proposition \ref{thm: saddle point} is straightforward, it has several notable consequences which we will describe in the coming sections.

\section{Maximizing variance on metric spaces.}
In this section, we will mostly (outside of Proposition \ref{cor: general V} below) focus on the case where $V$ is quadratic, $V(x,y) =d^2(x,y)$, where $X \subseteq M$, $Y =conv(X) \subseteq M$ and $(M,d)$ is a complete metric space.

We will call $R =\inf\{r:X\subseteq \bar B(y,r) \text{ for some } y \in M\} =\inf_{y \in M}\sup_{x\in X}d(x,y)$ the \emph{circumradius} of $X$ in $M$, where $\bar B(y,r)$ is the closed ball of radius $r$ centered at $y$.  If a ball $\bar B(y,r)$ attaining the minimum exists, we call it a \emph{circumball} of $X$ in $M$ and its center $y$ a \emph{circumcenter} of $X$ in $M$.
 
Recall that the Wasserstein metric on $P(M)$ is defined by
$$
W_2^2(\mu,\nu):=\inf\int_{M \times M}d^2(x,y)d\gamma(x,y)
$$
where the infimum is over all joint measures $\gamma \in P(M\times M)$ whose marginals are $\mu$ and $\nu$.  Of particular relevance here is the case where $\mu = \delta_x$ is a Dirac mass, in which case product measure $\delta_x \otimes \nu$ is the only joint measure with $\delta_x$ and $\nu$ as marginals and so $W_2^2(\delta_x ,\nu) =\int_Md^2(x,y)d\nu(y)$.

There is  a canonical isometry $i:X \rightarrow P(M)$, $i(x) =\delta_x$ from $X$ into $P(M)$ endowed with the Wasserstein distance.  We can then consider the circumradius of the subset  $i(X)$ of the metric space $(P(M), W_2)$, which is the infimum over all $R$ such that $W_2(\sigma, \nu) \leq R$ for all $\sigma \in i(X)$, for some $\nu \in P(M)$, or, equivalently, the infimum over $R$ such that,
$$
\int_Md^2(x,y)d\nu(y)=W_2^2(\delta_x ,\nu) \leq R^2
$$
for all $x \in X$, for some $\nu \in P(M)$.  We also note, as was observed in \cite{KimPass18}, that generalized barycenters $\nu$ of $\mu$ are exactly barycenters of the pushforward $i_\#\mu \in P(P(M))$, with respect to the Wasserstein metric. 

Our main result is that a generalization of Lim and McCann's Theorem 1.1 b) \cite{LimMcCann20a} holds for the canonical embedding of $X$  into its Wasserstein space $P(X)$:
\begin{theorem}\label{cor: general Jung}
Assume that both $X \subseteq M$ and its convex hull $Y=conv(X)$ are compact, and that each $\mu \in P(X)$ admits a barycenter.  Then there exists a circumcenter $\nu \in P(M)$ of $i(X)$ and the square $R^2$ of the circumradius $R$ of $i(X)$ is equal to the maximal variance of measures $\mu \in P(X)$, and to the minimal anti-variance among measures $\bar \nu \in P(Y)$.
	
	Furthermore, any $ \bar \mu$ maximizing the variance is supported on the corresponding sphere; that is $W_2(\delta_x,\nu)=R$ for $\mu$ a.e. $x$.  If $\bar \nu$ is any other circumcenter of $i(X)$, we also have $W_2(\delta_x,\bar \nu)=R$ for $\bar \mu$ almost every $x$.
\end{theorem}
\begin{proof}
	Let $(\mu, \nu)$ be the saddle point guaranteed to exist by Proposition \ref{thm: saddle point}.  Using Lemma \ref{lem: gen barycenters}, $\nu$ almost every $y$ is a barycenter of $\mu$ and the maximal variance is given by $R^2:=\int_Xd^2(x,y)d\mu(x)$ for each such $y$; integrating against $\nu$, and using the fact that $\mu$ is a generalized anti-barycenter of $\nu$, combined with Lemma \ref{lem: gen barycenters} again gives $R^2=\int_X\int_Yd^2(x,y)d\mu(x)d\nu(y) =\int_Yd^2(x,y)d\nu(y) = W^2_2(\delta_x,\nu)$ for $\mu$ almost all $x$.  This $R^2$ is also the anti-variance of $\nu$ (and therefore the minimal anti-variance by Proposition \ref{thm: saddle point}), which means that for every $z \in X$ we have $W^2_2(\delta_z,\nu)=\int_Yd^2(z,y)d\nu(y) \leq R^2$; that is, $i(X) \subseteq B(\nu,R)$.  The same argument applies to any other $\bar \mu$ maximizing the variance and $\bar \nu$ minimizing the anti-variance.
	  
	It remains to show that $R$ is the circumradius of $i(X)$ in $P(M)$; this is easily established by contradiction.  If $i(X) \subseteq \bar B(\nu',R')$ for some $\nu' \in P(M)$ and some $R'<R$,  then each $x \in X$ satisfies $\int_Yd^2(x,y)d\nu'(y) = W_2^2(\delta_x,\nu') \leq R'^2<R^2 $ and so the anti-variance of $\nu'$ is $\max_{x \in X}\int_Md^2(y,x)d\nu'(y)\leq R'^2<R^2 $.  This contradicts the fact that $\nu$ has minimal anti-variance.
\end{proof}

When barycenters are unique, as is the case on $CAT(0)$ spaces, or more generally, when the circumradius of $X$ is sufficiently small relative to its sectional curvature, this theorem yields the following modest contribution to our understanding of the geometry of Wasserstein space.
\begin{corollary}\label{cor: circumradius equal}
Let $X\subseteq M$ be such that  at least one variance maximizing probability measure $\mu \in P(X)$ has a unique barycenter.  Then, under the assumptions in Theorem \ref{cor: general Jung}, the circumradius of $X$ is the same as the circumradius of $i(X)$.  
\end{corollary}
\begin{proof}
	If $X\subseteq \bar B(y,R)$  then the ball $\bar B(\delta_y,R)$ clearly contains $i(X)$ in $P(M)$, and so the circumradius of $i(X)$ is always less than or equal to that of $X$.
	
	On the other hand, Proposition \ref{thm: saddle point} implies the that each variance maximizing measure $\mu$ concentrates on the boundary $\{x: d(x,y)=R\}$ of a circumball $\bar B(\nu,R)$ of $i(X)$, where $\nu$ is any generalized barycenter of $\mu$.
	But this means that $\nu$ almost every $y$ is a barycenter of $\mu$, which by our uniqueness  assumption  implies $\nu = \delta_y$ where $y$ is the unique barycenter of $\mu$.  Since $i(X) \subseteq \bar B(\delta_y,R)$, we clearly have $X \subseteq \bar B(y,R)$, meaning that $R$ is the circumradius of $X$.
\end{proof}

The following Corollary generalizes Proposition 3.1 in \cite{LangSchroeder97} from CAT($k$) space (for compact sets $X$; compactness was not required in \cite{LangSchroeder97}).  Note that the uniqueness of barycenters required here follows under their assumptions from \cite{Yokota16} and \cite{Yokota17}; however, our result also applies to other subsets of metric spaces without upper curvature bounds but on which barycenters are unique, such as closed subsets of the set $P_{ac}(M) \subseteq P(M)$ of absolutely continuous probability measures on compact underlying metric spaces, equipped with the Wasserstein metric \cite{KimPass2017}.

\begin{corollary}\label{cor: unique circumcenter}
	Assume that $X \subseteq M$ is such that at least one variance maximizing $\mu$ in $P(X)$ has a unique barycenter. Then, under the assumptions in Theorem \ref{cor: general Jung}, the circumcenter $y \in Y$ of $X$ is unique. 
	\end{corollary}
\begin{proof}
	 Since by Theorem \ref{cor: general Jung} and Corollary \ref{cor: circumradius equal}, the square $R^2$ of the circumradius is equal to the minimal anti-variance, we see that  $y$ is a circumcenter if and only if $\delta_y$ has minimal anti-variance.

	 Letting $\mu$ be the variance maximizer with a unique barycenter, Proposition \ref{thm: saddle point} and Lemma \ref{lem: gen barycenters} then imply that every circumcenter is a barycenter of $\mu$; uniqueness of the circumcenter then  follows.
\end{proof}
The following immediate consequence of Theorem \ref{cor: general Jung} and Corollaries \ref{cor: circumradius equal} and \ref{cor: unique circumcenter} directly generalizes Theorem 1.1 (b) of Lim and McCann \cite{LimMcCann20a}.
\begin{corollary}\label{cor: unique bc}
Assume that $X \subseteq M$ is such that at least one variance maximizing  $\mu$ in $P(X)$ has a unique barycenter.  Then, under the assumptions in Theorem \ref{cor: general Jung}, the maximal variance among measures in $P(X)$ is equal to the circumradius $R$ of $X$, and every variance maximizing $\mu$ gives full measure to the boundary $\partial B(y,R)$ of the circumball. 
\end{corollary}
In particular, it is worth noting that the maximal variance over any compact, convex set $X\subseteq \mathbb{R}^n$ is equal to the minimal anti-variance over the same set, and  anti-variance is minimized by the Dirac mass $\delta_y$, where $y$ is the circumcenter of $X$.

The assumption on the uniqueness of barycenters in Corollaries \ref{cor: circumradius equal}, \ref{cor: unique circumcenter} and \ref{cor: unique bc} is necessary, as the following example verifies. 
\begin{example} Let $X=Y=\mathcal{S}^n$ be the round sphere with metric diameter $1$ (ie, the geodesic distance from the north to south pole is $1$).  Then the only metric ball containing $X$ is the entire sphere, that is $\bar B(y,1)$ for any $y \in \mathcal{S}^n$. The circumradius is therefore $1$ and each point is a circumcenter.  However, it is not hard to see that there is no measure $\mu \in P(X)$ whose variance is $1$, so the maximal variance must be strictly  less than $1$.
	
On the other hand, Theorem \ref{cor: general Jung} still holds here. It is not hard to see that $(\mu,\nu ) =(vol/|vol|,vol/|vol|)$ is the saddle point from Proposition \ref{thm: saddle point}; that is both $\mu$ and $\nu$ are uniform.

The maximal variance is therefore attained by uniform measure and is equal to the average of the squared distance (which is clearly strictly less than the square $1$ of the circumradius, which is the \emph{maximal} squared distance between points).

 Theorem \ref{cor: general Jung} then implies that this coincides with the circumradius of $i(X)$, and that the circumcenter $\nu$ in Wasserstein space is uniform measure.  Symmetry implies $W_2(\delta_x,\nu) = \int_Xd^2(x,y)d\nu(y)$ is constant throughout $x \in spt(\mu) =X$, as predicted by Theorem \ref{cor: general Jung}.
\end{example}

We close this section by noting that Theorem \ref{cor: general Jung} and Corollary \ref{cor: unique bc} easily extend to general valuation, or moment functions, $V$.  The proof of the following Proposition is essentially the same as the arguments earlier in this section.

\begin{proposition}\label{cor: general V}
	Let $V: X\times Y \rightarrow \mathbb{R}$ be a continuous function on compact metric spaces $X$ and $Y$. Then
	\begin{enumerate}
		\item The maximal $V$-variance among measures $\mu \in P(X)$ is equal to the minimal $V$-anti-variance among $\nu \in P(Y)$.
		\item Let $\nu \in P(Y)$ minimize the $V$-anti-variance. Then $\mu \in P(X)$ maximizes the $V$-variance if and only for $\mu$ almost every $x$, $\int_YV(x,y) d\nu(y) =C_\nu$, where $C_\nu :=\max_{z \in X}\int_YV(z,y)d\nu(y)$ is the $V$-anti-variance of $\nu$ and $\nu$ almost every $y$ is a $V$-barycenter of $\mu$.		
		 \item 	If at least one variance maximizing measure $\mu$ on $X$ has a unique $V$-barycenter $y$, then the $V$-anti-variance is uniquely minimized by the Dirac mass $\delta_y$.
		 
		 In this case, $\bar \mu$ maximizes the $V$-variance if and only if $y$ is a $V$-barycenter of $\bar \mu$ and for $\bar \mu$ almost every $x$ we have $V(x,y) =C_y$, where $C_y:=\max_{z \in X}V(z,y)$ is the $V$-anti-variance of $\delta_y$.
	\end{enumerate}

\end{proposition}
\begin{example}
	Taking $V(x,y)=v(x-y)$ for $v$  convex and coercive on a compact $X \subseteq \mathbb{R}^n$ and compact $Y\subseteq \mathbb{R}^n$ large enough to contain all $V$-barycenters of measures on $X$, we recover the characterization of minimizers in \cite[Theorem 1.7]{LimMcCann20a}, provided that the $V$-anti-variance is minimized by a Dirac mass.
	
	This is clearly the case if $V$ is strictly convex, from the third point in the preceding Proposition; we claim it holds even without  strict convexity.  To see this, let $\nu$ minimize the anti-variance and $\bar y = \int_{\mathbb{R}^n}yd\nu(y)$ be the barycenter of $\nu$. Let $\bar x$ maximize  $x \mapsto V(x,\bar y)$, so that $V(\bar x,\bar y)=AVar_v(\delta_{\bar y})$.  By Jensen's inequality we have
	$$
AVar_v(\delta_{\bar y})	=V(\bar x, \bar y ) \leq \int_YV(\bar x ,y) d\nu(y) \leq AVar_V(\nu).
	$$
Thus, $\delta_{\bar y}$ minimizes the $V$-anti-variance, even for non-strictly convex $v$, and so Proposition \ref{cor: general V} implies Theorem 1.7 in \cite{LimMcCann20a}.  

In addition, the Proposition above applies to much more general $V$, with no convexity assumptions, than are dealt with in \cite{LimMcCann20a}, although in the general case the characterization of $V$-variance minimizers involves minimizers of the $V$-anti-variance of measures $\nu$, rather than minimizers of the simpler function $y \mapsto \max_xV(x,y)$, as is the case in \cite{LimMcCann20a}.

\end{example}
\section{Jung's theorem on spaces with curvature bounded above}

The purpose of this section is to provide a new proof of Jung's theorem, comparing the circumradius and diameter, $D=\sup_{x_0,x_1 \in X}d(x_1,x_2)$ of a compact set $X \subseteq M$,  when $M$ is a  CAT($k$) spaces.  We recall that CAT($k$) spaces are not-necessarily smooth generalizations of Riemannian manifolds with sectional curvature bounded above by $k$; we refer to, for example, \cite{AlexanderKapovitchPetrunin}, for precise definitions of CAT($k$) spaces and other relevant concepts.

The first proof of Jung's theorem on CAT($k$) spaces is due to Dekster \cite{Dekster97}; shortly afterwards, Lang and Schroeder produced a improved result with many technical assumptions removed \cite{LangSchroeder97}.

The formulation of Jung's theorem on CAT($k$) spaces requires the following notion:
\[ S(R,n,k) = \left\{
	\begin{array}{l l}
	\frac{1}{\sqrt{-k}}2\sinh^{-1}\big(\sqrt{\frac{n+1}{2n}}\sinh(\sqrt{-k}R)\big) & \quad k<0\\
	2R\sqrt{\frac{n+1}{2n}} & k=0\\
	\frac{1}{\sqrt{k}}2\sin^{-1}\big(\sqrt{\frac{n+1}{2n}}\sin(\sqrt{k}R)\big) &k>0.
	\end{array} \right.\]
We then set $S(R,\infty, K) = \lim_{n\rightarrow \infty} S(R,n,k) $, that is:
 
\[ S(R,\infty,k) = \left\{
\begin{array}{l l}
\frac{1}{\sqrt{-k}}2\sinh^{-1}\big(\frac{1}{\sqrt{2}}\sinh(\sqrt{-k}R)\big) & \quad k<0\\
\sqrt{2}R & k=0\\
\frac{1}{\sqrt{k}}2\sin^{-1}\big(\frac{1}{\sqrt{2}}\sin(\sqrt{k}R)\big) &k>0.
\end{array} \right.\]

We note that $S(R,n,k)$ can be interpreted as the third side length of a triangle in the model space of constant sectional curvature $k$ (that is, the $n$-dimensional complete, simply connected Riemannian manifold of constant sectional curvature $k$), when the other two sides both have length $R$ and subtend an angle of $\arccos(-1/n)$. 
  
We will need the following Lemma:
\begin{lemma}\label{lem: diff of distance}
Let  $t \mapsto x(t)$ be a unit speed geodesic  in the complete CAT($k$) space $M$ and $p \in M$ such that, if $k>0$, $d(x(0),p) <\frac{\pi}{2\sqrt{k}}$. Then $t \mapsto d^2(x(t),p)$ is differentiable  at $t=0$; its derivative is given by:

$$
\frac{d}{dt}(d^2(x(t),p))\Big |_{t=0} =-2d(x(0),p)\cos(\alpha)
$$
where $\alpha$ is the angle at $x(0)$ between $x(t)$ and the (unique) minimizing geodesic joining $x(0)$ and $p$.
\end{lemma}
\begin{proof}
	The first variation formula (see \cite[Formula 8.14.3]{AlexanderKapovitchPetrunin}) asserts that $t\mapsto d(x(t),p)$ is differentiable at $t=0$, with derivative $-\cos(\alpha)$; the result now follows from the chain rule.
\end{proof}
\begin{theorem}[Jung's Theorem on spaces with curvature bounded above]
Let $X \subseteq M$ be a compact subset of a complete CAT($k$) space $M$, which, if $k>0$, is contained in a metric ball of radius $r<\frac{\pi}{2\sqrt{k}}$.  Assume that all barycenters of measures on $X$ are contained in a compact subset $Y \subseteq M$.  Then there exists a unique circumcenter $y$ of $X$.

Furthermore, $D \geq S(R,\infty,k)$, where $D$ is the diameter of $X$ and $R$ the circumradius.  If, in addition, there exists a variance maximizer $\mu \in P(X)$ supported on $n+1$ points, we have $D \geq S(R,n,k)$. 
 \end{theorem}

\begin{remark}
	Lang and Schroeder's version did not require the compactness assumption on $X$ or $Y$, assuming instead only boundedness of the non-empty set $X$.  On the other hand, compactness seems to be necessary for our approach to work.  If $X=\{e_i\}_{i=1}^{\infty}$ is an orthonormal basis in a Hilbert space $M$, then the circumcenter $0$ is the origin, which is not a barycenter of any probability measure on $X$.  Therefore, there is no $\mu \in P(X)$ so that $(\mu,\delta_0)$ is a saddle point of $\mathcal{V}$, which is a key ingredient in our argument.
\end{remark}

\begin{proof}
The existence and uniqueness of barycenters in CAT($k$) spaces, due, under the conditions here, to Yokota \cite{Yokota17}\cite{Yokota16}, together with Corollary \ref{cor: unique circumcenter} imply the existence of a unique circumcenter, $y\in Y$ which is the barycenter of every variance maximizing $\mu \in P(X)$, which, in turn, are all supported on the boundary $\partial B(y,R)$ of the circumball.

Now fix a variance maximizing $\mu$ and let $y(\cdot):[0,R] \rightarrow M$ be a unit speed geodesic starting at $y(0)=y$.    By Lemma \ref{lem: diff of distance} and the dominated convergence theorem,  the function
\begin{equation}\label{eqn: variance function}
t \mapsto \int_Xd^2(y(t),x)d\mu(x)
\end{equation}
is differentiable at $t=0$ and its derivative is 

\begin{equation}\label{eqn: variance derivative}
\int_X-2d(y,x)\cos(\alpha(x))d\mu(x)=-2R\int_X\cos(\alpha(x))d\mu(x),
\end{equation}
where $\alpha(x)$ is the angle at $y=y(0)$ between $y(t)$ and the geodesic  joining $y$ and $x$. Since $y$ is a barycenter of $\mu$, the function \eqref{eqn: variance function} is minimized at the endpoint $t=0$ and so its derivative \eqref{eqn: variance derivative} must be non-negative.

Now, letting the endpoint $z:=y(R) \in \partial B(y,R)$ of the geodesic be any point in the support of $\mu$, this inequality means $\cos(\alpha(x)) \leq 0$ for some $x$ in the support of $\mu$, meaning that $|\alpha(x)|\geq \frac{\pi}{2}$.  Therefore, since the points $y$, $x$ and $z$ form a triangle in the $CAT($k$)$ space with the angle at $y$ equal to $\alpha(x) \geq \frac{\pi}{2}$, we have by \cite[8.3.1]{AlexanderKapovitchPetrunin}

$$
D \geq d(x,z) \geq d_k(R,R,\alpha(x)) \geq d_k(R,R,\frac{\pi}{2}) = S(R,\infty,k)
$$
where $d_K(l_1,l_2,\alpha)$ denotes the length of the third side of a triangle in the model space of constant curvature $k$ with two side lengths equal to $l_1$ and $l_2$ and an angle of $\alpha$ between them. 

Now, suppose further that there exists a variance maximizing $\mu$ supported on $n+1$ points, so that $\mu=\sum_{i=1}^{n+1}\lambda_i\delta_{x_i}$    with each $\lambda_i>0$ and $\sum_{i=1}^{n+1}\lambda_i=1$.  Assume without loss of generality that $\lambda_1=\max_i\lambda_i$ and choose the geodesic $y(\cdot)$ so that $y(R) =x_1$.  We then have
$$
\sum_{i=1}^{n+1}\lambda_i\cos(\alpha(x_i)) \leq 0,
$$
where each $\alpha(x_i)$ is the angle at $y=y(0)$ between the geodesic $y(t)$ from $y$ to $x_1$, and the geodesic from $y$ to $x_i$.  Let $C=\min_i\cos(\alpha(x_i)) \leq 0$.  Then the above yields
\begin{eqnarray*}
-\lambda_1 &\geq & \sum_{i=2}^{n+1}\lambda_i\cos(\alpha(x_i))\\
&\geq & C\sum_{i=2}^n\lambda_i\\
&\geq &Cn\lambda_1.
\end{eqnarray*}
Therefore, $C\leq -\frac{1}{n}$.  Assuming without loss of generality that $C=\cos(\alpha(x_2))$, we have, again by \cite[8.3.1]{AlexanderKapovitchPetrunin},
$$
D \geq d(x_1,x_2) \geq d_K(R,R,\arccos (-\frac{1}{n}))=S(R,n,k),
$$
as desired.
\end{proof}

\begin{remark}
	In general, the inequality $D \geq S(R,\infty,k)$ is the best we can hope for (note that the notion of CAT($k$) space does not come with a dimension). 
	 Dekster \cite{Dekster97} proves $D \geq S(R,n,k)$ under (a strengthening of) the additional assumption that  the intersection  $\omega( X \cap \partial B(y,R))$  is contained in a set isometric to the $(n-1)$ -- dimensional unit sphere (here $\omega( X)$ is the image of $X$ under the canonical mapping into the space of directions at $y$), as is always the case, for example, on an $n$-dimensional Riemannian manifold with sectional curvature less than $k$.  Under Dekster's assumption, Cartheodory's theorem implies that that $\mu$ can be taken to be supported on $n+1$ points, as hypothesized by Lang and Schroeder \cite{LangSchroeder97} and in our version above.
\end{remark}
\bibliographystyle{plain}
\bibliography{biblio}
\end{document}